\documentclass[12pt]{article}
\usepackage{amsmath,amscd,amsthm,a4,amssymb}

\newtheorem{theorem}{Theorem}[section]
\newtheorem{lemma}[theorem]{Lemma}
\newtheorem{corollary}[theorem]{Corollary}
\newtheorem{definition}[theorem]{Definition}
\newtheorem{remark}[theorem]{Remark}

\numberwithin{equation}{section}

\usepackage[]{graphicx}

\def\Rn{{\mathbb{R}^n}}

\def\a {\alpha}

\def\L1loc{L_{\Phi}^{\rm loc}(\Rn)}

\usepackage{amsmath,amscd,amsthm,amssymb}
\usepackage{color}

\begin{document}

\begin{center}
\Large Some notes on commutators of the fractional maximal and Riesz potential operators on Orlicz spaces
\end{center}

\

\centerline{\large Vagif S. Guliyev$^{a}$, Fatih Deringoz$^{b,}$\footnote{
		Corresponding author.
		\\
		E-mail adresses: vagif@guliyev.com (V.S. Guliyev), deringoz@hotmail.com (F. Deringoz), sabhasanov@gmail.com (S.G. Hasanov).}, Sabir G. Hasanov$^{c}$ }

\

\centerline{$^{a}$\it Institute of Applied Mathematics, Baku State University, Baku, AZ 1148 Azerbaijan}



\centerline{$^{b}$\it Department of Mathematics, Ahi Evran University, Kirsehir, Turkey}

\centerline{$^{c}$\it Ganja State University, Ganja, Azerbaijan}

\

\

\begin{abstract}
The main focus of this paper is commutators and maximal commutators on Orlicz spaces for fractional maximal functions and Riesz potential. The main advance in comparison with the existing results is that we manage to obtain conditions for the boundedness in less restrictive terms.
\end{abstract}

\



\

\section{Introduction}\label{s1}

Norm inequalities for several classical operators of harmonic analysis have been widely studied in the context of Orlicz spaces. It is well known that many of such operators fail to have continuity properties when they act between certain Lebesgue spaces and, in some situations, the Orlicz spaces appear as adequate substitutes.

The fractional maximal operator $M_{\a}$, $0\le\a<n$ and the Riesz potential operator $I_{\a}$, $0<\a<n$ are  defined by
$$
M_{\a} f(x)=\sup_{t>0}|B(x,t)|^{-1+ \frac{\a}{n}}\int _{B(x,t)} |f(y)|dy,~~
I_{\a} f(x)=\int _{\Rn} \frac{f(y)}{|x-y|^{n-\a}}dy.
$$
Here and everywhere in the sequel $B(x,r)$ is the ball in $\Rn$ of radius $r$ centered
at $x$ and  $|B(x,r)|=v_n r^n$ is its Lebesgue measure, where $v_n$ is the volume of the unit ball in $\Rn$.
If $\a=0$, then $M \equiv M_{0}$ is the well known Hardy-Littlewood maximal operator. Recall that, for $0<\alpha<n$,
\begin{equation}\label{eq001pws}
	M_{\alpha}f(x)\leq v_n^{\frac{\a}{n}-1} \, I_{\a}(|f|)(x).
\end{equation}

We also define the sharp maximal function as
\begin{align*}
	M^{\sharp} f(x)=\sup_{t>0}|B(x,t)|^{-1}\int _{B(x,t)} |f(y)-f_{B(x,t)}|dy,
\end{align*}
where $f_{B(x,t)} = |B(x,t)|^{-1} \int _{B(x,t)} f(y) dy$.

The commutators generated by  a suitable function $b$ and the operators $M_{\a}$ and  $I_{\a}$ are formally defined by
\begin{equation*}
	[b,M_{\a}]f = b \, M_{\a}f - M_{\a}(bf),~~
	[b,I_{\a}]f = b \, I_{\a}f - I_{\a}(bf),
\end{equation*}
respectively.

Given a measurable function $b$ the operators $M_{b,\a}$ and $|b,I_{\a}|$ are defined by
\begin{equation*}
	M_{b,\a}f(x)=\sup_{t>0}|B(x,t)|^{-1+\frac{\a}{n}} \, \int _{B(x,t)}|b(x)-b(y)||f(y)|dy
\end{equation*}
and
\begin{equation*}
	|b,I_{\a}|f(x)=\int _{\Rn}\frac{|b(x)-b(y)|}{|x-y|^{n-\a}}f(y)dy,
\end{equation*}
respectively. If $\alpha=0$, then $M_{b,0}\equiv M_{b}$ is called the maximal commutator. Recall that, for $0<\alpha<n$,
\begin{equation} \label{fgs01}
	M_{b,\a}f(x) \leq v_n^{\frac{\a}{n}-1} \, |b,I_{\a}|(|f|)(x)
\end{equation}
and
\begin{equation*}
	\big|[b,I_{\a}]f(x)\big| \leq  |b,I_{\a}|(|f|)(x).
\end{equation*}

Suppose that $b\in L^1_{\rm loc}(\Rn)$. Then $b$ is said to be in $BMO(\Rn)$ if the seminorm given by
\begin{equation*}
	\|b\|_{\ast}=\sup_{x\in\Rn, r>0}\frac{1}{|B(x,r)|}
	\int_{B(x,r)}|b(y)-b_{B(x,r)}|dy
\end{equation*}
is finite.

The following theorem is valid.
\begin{theorem}\label{rem01} \cite{Chanillo, DGSDebr2017}
	Let $0<\a <n$, $1< p < \frac{n}{\a}$ and $\frac{1}{q}=\frac{1}{p}-\frac{\a}{n}$. Then $M_{b,\a}$, $[b,I_{\a}]$ and $|b,I_{\a}|$ are bounded operators from $L^p(\Rn)$ to $L^q(\Rn)$ if and only if $b\in  BMO(\Rn)$.
\end{theorem}

The proof of Theorem \ref{rem01} for $[b,I_{\a}]$ was given in \cite{Chanillo} and for $M_{b,\a}$ and $|b,I_{\a}|$ was given in \cite{DGSDebr2017}.
The problem of the boundedness of those commutators in Orlicz spaces was investigated in papers \cite{GulDerHasJIA,GulDerHasMathN} by the authors. Here we solve this problem with a different approach and refine our previous results. More precisely, we manage to obtain the boundedness of $M_{b,\a}$ and $[b,I_{\a}]$ in Orlicz spaces under weaker conditions in comparison with the existing results. We also improve our results on boundedness of the Riesz potential in the Orlicz spaces and the corresponding weak-type version.

By $A \lesssim B$ we mean that $A \le C B$ with some positive constant $C$
independent of appropriate quantities. If $A \lesssim B$ and $B \lesssim A$, we
write $A\approx B$ and say that $A$ and $B$ are  equivalent.

\section{Preliminaries}
First, we recall the definition of Young functions.
\begin{definition}\label{def2} A function $\Phi : [0,\infty) \rightarrow [0,\infty]$ is called a Young function if $\Phi$ is convex, left-continuous, $\lim\limits_{r\rightarrow +0} \Phi(r) = \Phi(0) = 0$ and $\lim\limits_{r\rightarrow \infty} \Phi(r) = \infty$.
\end{definition}
From the convexity and $\Phi(0) = 0$ it follows that any Young function is increasing.
If there exists $s \in  (0,\infty )$ such that $\Phi(s) = \infty $,
then $\Phi(r) = \infty $ for $r \geq s$.
The set of  Young  functions such that
\begin{equation*}
	0<\Phi(r)<\infty \qquad \text{for} \qquad 0<r<\infty
\end{equation*}
will be denoted by  $\mathcal{Y}.$
If $\Phi \in  \mathcal{Y}$, then $\Phi$ is absolutely continuous on every closed interval in $[0,\infty )$
and bijective from $[0,\infty )$ to itself.

For a Young function $\Phi$ and  $0 \leq s \leq \infty $, let
$$\Phi^{-1}(s)=\inf\{r\geq 0: \Phi(r)>s\}.$$
If $\Phi \in  \mathcal{Y}$, then $\Phi^{-1}$ is the usual inverse function of $\Phi$.

It is well known that
\begin{equation}\label{2.3}
	r\leq \Phi^{-1}(r)\widetilde{\Phi}^{-1}(r)\leq 2r \qquad \text{for } r\geq 0,
\end{equation}
where $\widetilde{\Phi}(r)$ is defined by
\begin{equation*}
	\widetilde{\Phi}(r)=\left\{
	\begin{array}{ccc}
		\sup\{rs-\Phi(s): s\in  [0,\infty )\}
		& , & r\in  [0,\infty ), \\
		\infty &,& r=\infty .
	\end{array}
	\right.
\end{equation*}

A Young function $\Phi$ is said to satisfy the
$\Delta_2$-condition, denoted also as   $\Phi \in  \Delta_2$, if
$$
\Phi(2r)\le C\Phi(r), \qquad r>0
$$
for some $C>1$. If $\Phi \in  \Delta_2$, then $\Phi \in  \mathcal{Y}$. A Young function $\Phi$ is said to satisfy the $\nabla_2$-condition, denoted also by  $\Phi \in  \nabla_2$, if
$$
\Phi(r)\leq \frac{1}{2C}\Phi(Cr),\qquad r\geq 0
$$
for some $C>1$.

Let $L^0({\mathbb R}^n)$ be the set of all measurable functions.
\begin{definition} (Orlicz Space). For a Young function $\Phi$, the set
	$$L^{\Phi}(\Rn)=\left\{f\in  L^0(\Rn): \int _{\Rn}\Phi(k|f(x)|)dx<\infty
	\text{ for some $k>0$  }\right\}$$
	is called Orlicz space.
	
\end{definition}
$L^{\Phi}(\Rn)$ is a Banach space with respect to the norm, which is called Luxemburg-Nakano norm
$$\|f\|_{L^{\Phi}}=\inf\left\{\lambda>0:\int _{\Rn}\Phi\Big(\frac{|f(x)|}{\lambda}\Big)dx\leq 1\right\}.
$$

If $\Phi(r)=r^{p},\, 1\le p<\infty $, then $L^{\Phi}(\Rn)=L^{p}(\Rn)$. If $\Phi(r)=0,\,(0\le r\le 1)$ and $\Phi(r)=\infty ,\,(r> 1)$, then $L^{\Phi}(\Rn)=L^\infty (\Rn)$.

\begin{definition} For a Young function $\Phi$, the weak Orlicz space
	$$
	{\rm W}L^{\Phi}(\mathbb{R}^{n})=\{f\in L^{0}(\mathbb{R}^{n}):\Vert f\Vert_{{\rm W}L^{\Phi}}<\infty\}
	$$
	is defined by the quasi-norm
	$$
	\Vert f\Vert_{{\rm W}L^{\Phi}}=\sup_{\lambda>0} \| \lambda\, \chi_{_{(\lambda,\infty)}}(|f|)\|_{L^{\Phi}}.
	$$
\end{definition}

By elementary calculations we have the following property.
\begin{lemma}\label{charorlc}
	Let $\Phi$ be a Young function and $B$ be a set in $\mathbb{R}^n$ with finite Lebesgue measure. Then
	\begin{equation*}
		\|\chi_{_B}\|_{L^{\Phi}} = \|\chi_{_B}\|_{WL^{\Phi}}=\frac{1}{\Phi^{-1}\left(|B|^{-1}\right)}.
	\end{equation*}
\end{lemma}

We have the following scaling law :
\begin{lemma}\label{Vak03} \cite[Lemma 10]{DGNSShiPos2019} Let $\beta > 0$, $\Phi$ be a Young function and $\Phi_{\beta}(t) = \Phi(t^{1/\beta})$. Then
	\begin{align*}
		\big\| |f|^{\beta} \big\|_{L^{\Phi_{\beta}}} = \| f \|_{L^{\Phi}}^{\beta} ~~ \mbox{and} ~~
		\big\| |f|^{\beta} \big\|_{WL^{\Phi_{\beta}}} = \| f \|_{WL^{\Phi}}^{\beta}
	\end{align*}	
	for all measurable functions $f$.
\end{lemma}

The following results are
generalized H\"older's inequalities for the Orlicz and weak Orlicz spaces.
\begin{lemma}\label{oneilhold} \cite{ONeil1965}
	Let $\Phi_{1},\ \Phi_{2},\ \Phi_{3}$ be Young functions. If for all $t\geq 0$ we have
	\begin{center}
		$\Phi_{1}^{-1}(t)\Phi_{2}^{-1}(t)\leq\Phi_{3}^{-1}(t)$,
	\end{center}
	then
	\begin{center}
		$\Vert fg \Vert_{L^{\Phi_{3}}}\lesssim\Vert f\Vert_{L^{\Phi_{1}}}\Vert g\Vert_{L^{\Phi_{2}}}$
	\end{center}
	for every $f\in L^{\Phi_{1}}(\Rn)$ and $g \in L^{\Phi_{2}}(\Rn)$.
\end{lemma}

\begin{lemma}\label{oneilholdw} \cite{KaNa18}
	Let $\Phi_{1},\ \Phi_{2},\ \Phi_{3}$ be Young functions. If for all $t\geq 0$ we have
	\begin{center}
		$\Phi_{1}^{-1}(t)\Phi_{2}^{-1}(t)\leq\Phi_{3}^{-1}(t)$,
	\end{center}
	then
	\begin{center}
		$\Vert fg \Vert_{WL^{\Phi_{3}}}\lesssim\Vert f\Vert_{WL^{\Phi_{1}}}\Vert g\Vert_{WL^{\Phi_{2}}}$
	\end{center}
	for every $f\in WL^{\Phi_{1}}(\Rn)$ and $g \in WL^{\Phi_{2}}(\Rn)$.
\end{lemma}

\section{Commutators of Riesz potential and fractional maximal commutators in Orlicz spaces}

$~~~$ In this section we find necessary and sufficient conditions for the boundedness of
the operators $M_{b,\a}$ and $|b,I_{\a}|$ in Orlicz spaces.

The following results completely characterize the boundedness of $M_{\a}$ and $I_{\a}$ on Orlicz spaces.
\begin{theorem}\label{AdamsFrMaxCharOrl}\cite{GulDerHasAAMP}
	Let $0< \a<n$ and $\Phi, \Psi$ be Young functions. The condition
	\begin{equation}\label{adRieszCharOrl2}
		r^{-\frac{\alpha}{n}} \, \Phi^{-1}(r) \le C \Psi^{-1}(r)
	\end{equation}
	for all $r>0$, where $C>0$ does not depend on $r$, is necessary and sufficient for the boundedness of $M_{\a}$ from $L^{\Phi}(\Rn)$ to $WL^{\Psi}(\Rn)$. Moreover, if $\Phi\in\nabla_2$, the condition  \eqref{adRieszCharOrl2} is necessary and sufficient for the boundedness of $M_{\a}$ from $L^{\Phi}(\Rn)$ to $L^{\Psi}(\Rn)$.
\end{theorem}

\begin{remark}
	Theorem \ref{AdamsFrMaxCharOrl} was proved in \cite{GulDerHasAAMP} under the condition $\Phi\in\mathcal{Y}$. But now we know that the condition $\Phi\in\mathcal{Y}$ is superfluous, see \cite[Theorem 4]{DGNSShiPos2019}.
\end{remark}

\begin{theorem}\label{RieszPotCharOrl0V}
	Let $0< \a<n$ and $\Phi, \Psi$ be Young functions.
	Then the condition \eqref{adRieszCharOrl2} is necessary and sufficient for the boundedness of $I_{\a}$ from $L^{\Phi}(\Rn)$ to $WL^{\Psi}(\Rn)$. Moreover, if $\Phi\in\nabla_2$, the condition  \eqref{adRieszCharOrl2} is necessary and sufficient for the boundedness of $I_{\a}$ from $L^{\Phi}(\Rn)$ to $L^{\Psi}(\Rn)$.
\end{theorem}
\begin{proof}
	Suppose that $\varepsilon>0$ satisfies $0<\a-\varepsilon<\a+\varepsilon<n$ and $x \in \Rn$. Then
	\begin{align*}
		I_{\a}(|f|)(x) \le C \, \big(M_{\a-\varepsilon}f(x)\big)^{\frac{1}{2}}
		\, \big(M_{\a+\varepsilon}f(x)\big)^{\frac{1}{2}},
	\end{align*}
	where $C$ depends only on $\varepsilon$, $\a$, $n$.
	This is known as the Welland inequality (see inequality (2.3) in \cite{Welland}). Thus, by the boundedness properties of $M_{\a}$ we can derive boundedness results for $I_{\a}$.
	
	Now define
	$ A^{-1}(r)\approx r^{-\frac{\a-\varepsilon}{n}} \, \Phi^{-1}(r)  $ and  $ B^{-1}(r)\approx r^{-\frac{\a+\varepsilon}{n}} \, \Phi^{-1}(r)  $.
	Then by Theorem \ref{AdamsFrMaxCharOrl} we have $M_{\a-\varepsilon} : L^{\Phi}(\Rn) \to L^{A}(\Rn) $ and $M_{\a+\varepsilon} : L^{\Phi}(\Rn) \to L^{B}(\Rn) $. Note that by using the condition \eqref{adRieszCharOrl2} we get
	$$
	A^{-1}(r) \, B^{-1}(r)\approx \left[r^{-\frac{\alpha}{n}} \, \Phi^{-1}(r)\right]^{2}\lesssim \Psi^{-1}(r)^{2}=\Psi_{2}^{-1}(r),\text{   where  } \Psi_{2}(t) = \Psi(t^{1/2})
	$$
	which allows us to use Lemma \ref{oneilhold}. Now applying Lemma \ref{oneilhold}
	we get that the condition \eqref{adRieszCharOrl2} is sufficient for the boundedness of $I_{\a}$ from $L^{\Phi}(\Rn)$ to $L^{\Psi}(\Rn)$. Indeed, by Lemmas \ref{Vak03} and \ref{oneilhold}
	\begin{align*}
		\|I_{\a} f \|_{L^{\Psi}} & \le \|I_{\a} |f|\|_{L^{\Psi}} =
		\|\big(I_{\a} |f|\big)^2\|_{L^{\Psi_2}}^{\frac{1}{2}}
		\\
		& \lesssim \|M_{\a-\varepsilon}f \, M_{\a+\varepsilon}f\|_{L^{\Psi_2}}^{\frac{1}{2}}
		\\
		& \lesssim \|M_{\a-\varepsilon}f\|_{L^{A}}^{\frac{1}{2}} \,
		\|M_{\a+\varepsilon}f\|_{L^{B}}^{\frac{1}{2}}
		\\
		& \lesssim \|f\|_{L^{\Phi}}^{\frac{1}{2}} \, \|f\|_{L^{\Phi}}^{\frac{1}{2}}
		= \|f\|_{L^{\Phi}}.
	\end{align*}
	
	The proof of the boundedness  of the weak type is similar to the strong case, but now using the weak boundedness of $M_\alpha$ and Lemma \ref{oneilholdw}.
	
	The necessity of the condition \eqref{adRieszCharOrl2} follows from Theorem \ref{AdamsFrMaxCharOrl} with the help of pointwise inequality \eqref{eq001pws}.
\end{proof}
\begin{remark}
	Note that, Theorem \ref{RieszPotCharOrl0V} was proved in \cite[Theorem 3.3]{GulDerHasJIA}, if in addition the condition
	\begin{equation}\label{adRieszCharOrl02V}
		\int_{r}^{\infty} \Phi^{-1}\big(t^{-n}\big) t^{\alpha}\frac{dt}{t} \le C r^{\alpha} \Phi^{-1}\big(r^{-n}\big)
	\end{equation}
	for all $r>0$, where $C>0$ does not depend on $r$, is satisfied. In the proof of Theorem \ref{RieszPotCharOrl0V}  we do not use the condition \eqref{adRieszCharOrl02V}.
\end{remark}

For proving our main results, we need the following estimate.

\begin{lemma}\label{Vak01} \cite[Proposition 6.4.]{shiarainakai} Let $\Phi$ be a Young function with  $\Phi \in  \Delta_2\cap\nabla_2$, then
	\begin{align*}
		\|M f \|_{L^{\Phi}} \lesssim  \|M^{\sharp} f \|_{L^{\Phi}}
	\end{align*}
	for all $f \in L^{\Phi}(\Rn)$.
\end{lemma}

The next lemma has been proven in \cite{Shirai2006} in the case of $[b,I_{\a}]$ but the proof
works without changes also for the operator $|b,I_{\a}|$.
\begin{lemma}\label{Vak02}  \cite[Lemma 4.2]{Shirai2006} Let $0<\a<n$, $1<r<\infty$ and $b \in BMO(\Rn)$. Then for almost all $x \in \Rn$ and for every $f \in C_{0}^{\infty}(\Rn)$  we have pointwise estimates :
	\begin{align*}
		M^{\sharp} \big([b,I_{\a}] f\big)(x) & \lesssim \|b\|_{\ast} \, \big(I_{\a}|f|(x) +
		I_{\a r}(|f|^r)(x)^{1/r}\big),
		\\
		M^{\sharp} \big(|b,I_{\a}| f\big)(x) & \lesssim \|b\|_{\ast} \, \big(I_{\a}|f|(x) +
		I_{\a r}(|f|^r)(x)^{1/r}\big).
	\end{align*}
\end{lemma}

The following theorem gives necessary and sufficient conditions for the
boundedness of the operator $|b,I_{\a}|$ from $L^{\Phi}(\Rn)$ to $L^{\Psi}(\Rn)$.
\begin{theorem} \label{CommRieszCharOrl01}
	Let $0 < \a<n$, $b\in BMO(\Rn)$ and $\Phi, \Psi$ be Young functions such that $\Phi \in  \Delta_2\cap\nabla_2$ and $\Psi \in  \Delta_2\cap\nabla_2$.
	Then the condition \eqref{adRieszCharOrl2} is necessary and sufficient for the boundedness of $|b,I_{\a}|$ from $L^{\Phi}(\Rn)$ to $L^{\Psi}(\Rn)$.
\end{theorem}
\begin{proof}
	\
	
	\textit{Sufficiency:}  Let $1<r<\infty$. Applying Lemmas \ref{Vak01} and \ref{Vak02}, we see that for any $b \in BMO(\Rn)$ and $f \in C_{0}^{\infty}(\Rn)$
	\begin{align}\label{Vak04}
		\||b,I_{\a}| f \|_{L^{\Psi}} & \le  \|M (|b,I_{\a}| f) \|_{L^{\Psi}}  \lesssim  \|M^{\sharp} (|b,I_{\a}| f) \|_{L^{\Psi}}    \notag
		\\
		& \lesssim \|b\|_{\ast} \, \big\|I_{\a}|f| +
		I_{\a r}(|f|^r)^{1/r}\big\|_{L^{\Psi}}
		\\
		& \lesssim \|b\|_{\ast} \, \Big(\big\|I_{\a}|f|\big\|_{L^{\Psi}} +
		\big\|I_{\a r}(|f|^r)\big\|_{L^{\Psi_r}}^{1/r}\Big). \notag
	\end{align}
	By Theorem \ref{RieszPotCharOrl0V},  $I_{\a} : L^{\Phi}(\Rn) \to L^{\Psi}(\Rn)$  is bounded so let us show boundedness of $I_{\a r} : L^{\Phi_{r}}(\Rn) \to L^{\Psi_{r}}(\Rn)$, where $\Phi_{r}(t) = \Phi(t^{1/r})$. ). It is easy to check that $\Psi_{r}^{-1}(t) = \Psi^{-1}(t)^{r}$. Then, it immediately follows that
	\begin{align*}
		\Psi_{r}^{-1}(t) & = \Psi^{-1}(t)^{r} \gtrsim\, t^{-\frac{\alpha r}{n}} \, \Phi^{-1}(t)^{r} =
		t^{-\frac{\alpha r}{n}} \, \Phi_{r}^{-1}(t).
	\end{align*}
	Theorem \ref{RieszPotCharOrl0V} shows that the operator $I_{\a r}$ is bounded from $L^{\Phi_r}(\Rn)$ from $L^{\Psi_r}(\Rn)$  and likewise
	\begin{align*}
		\big\|I_{\a r}(|f|^r)\big\|_{L^{\Psi_r}}^{1/r} \lesssim \big\||f|^r \big\|_{L^{\Phi_r}}^{1/r}.
	\end{align*}
	Coming back to \eqref{Vak04} and definition of Luxemburg-Nakano norm, we have
	\begin{align*}
		\||b,I_{\a}| f \|_{L^{\Psi}} & \lesssim  \|b\|_{\ast} \, \Big(\|f\|_{L^{\Phi}} +
		\big\||f|^r \big\|_{L^{\Phi_{r}}}^{1/r}\Big) \lesssim  \|b\|_{\ast} \, \|f\|_{L^{\Phi}}.
	\end{align*}
	Since $\Phi\in \Delta_2$, we know that $C_{0}^{\infty}(\Rn)$ is dense in $L^{\Phi}(\Rn)$
	\cite[Corollary 3.7.10]{HarjulehtoHasto}. Now we get the result for all $f \in L^{\Phi}(\Rn)$  since the commutator is linear and bounded on a dense set.

	\textit{Necessity:} It follows from Theorem \ref{AdamsCommFrCharOrl} (2) since the boundedness of $|b,I_{\a}|$ implies the boundedness of $M_{b,\a}$ by pointwise inequality \eqref{fgs01}.
	
\end{proof}
\begin{corollary}
	Let $0 < \a<n$, $b\in BMO(\Rn)$ and $\Phi,\Psi$ be Young functions with such that $\Phi \in  \Delta_2\cap\nabla_2$ and $\Psi \in  \Delta_2\cap\nabla_2$. If the condition \eqref{adRieszCharOrl2} holds, then the operator $[b,I_{\a}]$ is bounded from $L^{\Phi}(\Rn)$ to $L^{\Psi}(\Rn)$.
\end{corollary}
\begin{remark}
	Note that, Theorem \ref{CommRieszCharOrl01} was proved in \cite[Theorem 5.4]{GulDerHasJIA}, if in addition the condition
	\begin{equation}\label{adRieszCommCharOrl03pot}
		\int_{r}^{\infty} \Big(1+\ln \frac{t}{r}\Big) \Phi^{-1}\big(t^{-n}\big) t^{\alpha} \frac{dt}{t} \le C r^{\alpha} \Phi^{-1}\big(r^{-n}\big)
	\end{equation}
	for all $r>0$, where $C>0$ does not depend on $r$, is satisfied. In the proof of Theorem \ref{CommRieszCharOrl01}  we do not use the condition \eqref{adRieszCommCharOrl03pot}.
\end{remark}

The following theorem gives necessary and sufficient conditions for the
boundedness of the operator $M_{b,\a}$ from $L^{\Phi}(\Rn)$ to $L^{\Psi}(\Rn)$.
\begin{theorem} \label{AdamsCommFrCharOrl}
	Let $0 < \a<n$, $b\in BMO(\Rn)$ and $\Phi, \Psi$ be Young functions such that $\Phi \in  \Delta_2\cap\nabla_2$ and $\Psi \in  \Delta_2\cap\nabla_2$. Then the condition \eqref{adRieszCharOrl2}
	is necessary and sufficient for the boundedness of $M_{b,\a}$ from $L^{\Phi}(\Rn)$ to $L^{\Psi}(\Rn)$.
\end{theorem}
\begin{proof}
	
	\
	
	\textit{Sufficiency:} Let $0<\a<n$, $b\in BMO(\Rn)$ and $\Phi, \Psi$ be Young functions. Since we have the
	pointwise estimate $M_{b,\a} f(x) \lesssim |b,I_{\a}|(|f|)(x)$, $x \in \Rn$, Theorem \ref{CommRieszCharOrl01} immediately yields
	\begin{align*}
		\|M_{b,\a} f\|_{L^{\Psi}} & \lesssim  \||b,I_{\a}|(|f|)\|_{L^{\Psi}} \lesssim  \|b\|_{\ast} \, \|f\|_{L^{\Phi}}.
	\end{align*}
	
	\textit{Necessity:} Suppose that $b\in BMO(\Rn) \setminus {\{const\}}$ and the fractional maximal commutator $M_{b,\alpha}$ is bounded from $L^{\Phi}(\Rn)$ to $L^{\Psi}(\Rn)$.
	Then
	\begin{align} \label{Vag01}
		\|M_{b,\alpha} f\|_{L^{\Psi}} \lesssim \|b\|_{\ast} \, \|f\|_{L^{\Phi}}
	\end{align}
	
	In particular, for $b(\cdot):=\ln|\cdot|\in BMO(\Rn)$ and $f:=\chi_{B}$ for all ball $B$, we have
	\begin{align} \label{Vag02}
		\|M_{b,\alpha} \chi_{B}\|_{L^{\Psi}} \le \|b\|_{\ast} \, \|\chi_{B}\|_{L^{\Phi}}.
	\end{align}
	Taking in \eqref{Vag02} $B = B(0,r)$ with $r = (a_1 uv)^{-\frac{1}{n}}$, where $a_r = |B(0,r)|$, $u>0$ and $v>1$, we get
	\begin{align*}
		\|\chi_{B(0,r)}\|_{L^{\Phi}} & = \frac{1}{\Phi^{-1}\big(|B(0,r)|^{-1}\big)} = \frac{1}{\Phi^{-1}\big( a_1^{-1} r^{-n}\big)}
		\\
		& = \frac{1}{\Phi^{-1}(uv)} \le \frac{1}{uv} \, \widetilde{\Phi}^{-1}(uv) .
	\end{align*}
	
	On the other hand, if $ x\notin B(0,r) $ then $ B(0,r)\subset B(x,2|x|) $ because for any $ y\in B(0,r) $ we have
	\begin{align*}
		|x-y|\leq|x|+|y|\leq|x|+r\leq 2|x|.
	\end{align*}
	Also for each $ y\in B(0,r) $, we have
	$$
	b(x)-b(y)\geq \ln\Big(\frac{|x|}{r}\Big).
	$$
	Therefore
	\begin{align*}
		M_{b,\a}\chi_{B(0,r)}(x) & \geq \dfrac{1}{|B(x,2|x|)|^{{1-\frac{\a}{n}}}} \int_{B(x,2|x|)\cap B(0,r)}|b(x)-b(y)|~d(y)
		\\
		& \geq a_{1}^{\frac{\alpha}{n}} r^{n} \, \Big( \dfrac{1}{2|x|}\Big)^{n-\a} \ln\Big(\frac{|x|}{r}\Big).
	\end{align*}
	For $g:=\widetilde{\Phi}^{-1}(u)\chi_{B(0,s)}$ with
	$s:=(a_1u)^{-\frac{1}{n}}$ we obtain
	\begin{align*}
		\int_{\Rn}\widetilde{\Phi}(|g(x)|)~dx = u\, |B(0,s)| = u \, s^{n} \, |B(0,1)| = 1.
	\end{align*}
	
	Since the Luxemburg-Nakano norm is equivalent to the Orlicz norm
	$$
	\|f\|_{L^{\Phi}}^{0} := \sup\left\{\int_{\Rn}|f(x)g(x)|~dx : \|g\|_{L^{\widetilde{\Phi}}} \leq 1\right\}
	$$
	(more precisely, $\|f\|_{L^{\Phi}}\leq \|f\|_{L^{\Phi}}^{0} \leq 2 \|f\|_{L^{\Phi}}$), it follows that
	\begin{align*}
		&\left\| M_{b,\a}\chi_{B(0,r)}\right\|_{L^{\Psi}}^{0}
		\\
		&=\sup\left\lbrace \int_{\Rn} M_{b,\a}\chi_{B(0,r)}(x) \, |g(x)|~dx : \int_{\Rn}
		\widetilde{\Psi}(|g(x)|)~dx \leq 1 \right\rbrace
		\\
		& \geq \widetilde{\Psi}^{-1}(u) \int_{B(0,s)} M_{b,\a}\chi_{B(0,r)}(x)~dx
		\\
		& \geq a_{1}^{\frac{\alpha}{n}} r^{n} \, \widetilde{\Psi}^{-1}(u) \int_{B(0,s)\setminus B(0,r)}
		\Big( \dfrac{1}{2|x|}\Big)^{n-\a} \ln\Big(\frac{|x|}{r}\Big)~dx
		\\
		&=\dfrac{a_{1}^{\frac{\alpha}{n}} \widetilde{\Psi}^{-1}(u)}{2^{n-\a} a_1uv} \int_{r<|x|<s}\dfrac{1}{|x|^{n-\a}}\ln\Big(\frac{|x|}{r}\Big)~dx  ~~~(\mbox{using spherical coordinates})
		\\
		& = \dfrac{\widetilde{\Psi}^{-1}(u)u^{-\frac{\alpha}{n}}}{\alpha^{2}2^{n-\alpha} ~uv}(\alpha \ln v+n v^{-\frac{\alpha}{n}}-n).
	\end{align*}
	Hence, \eqref{Vag02} implies that
	\begin{align*}
		\dfrac{\widetilde{\Psi}^{-1}(u)u^{-\frac{\alpha}{n}}}{\alpha^{2}2^{n-\alpha} ~uv}(\alpha \ln v+n v^{-\frac{\alpha}{n}}-n)  \le \frac{\|\ln|\cdot|\|_{\ast}}{uv} \, \widetilde{\Phi}^{-1}(uv)
	\end{align*}
	for $u>0$ and $v>1$. Thus, taking $v=2 $ we obtain for $ u>0 $
	$$
	u^{-\frac{\a}{n}} \, \widetilde{\Psi}^{-1}(u) \lesssim \widetilde{\Phi}^{-1}(u)
	$$
	or
	$$
	u^{-\frac{\a}{n}} \, \Phi^{-1}(u)  \lesssim \, \Psi^{-1}(u).
	$$
	
\end{proof}
\begin{remark}
	Note that, Theorem \ref{AdamsCommFrCharOrl} was proved in \cite[Theorem 3.8]{GulDerHasMathN}, if in addition the condition
	\begin{equation}\label{adFrCommCharOrl3}
		\sup_{r<t<\infty} \Big(1+\ln \frac{t}{r}\Big) \Phi^{-1}\big(t^{-n}\big) t^{\alpha} \le C r^{\alpha} \Phi^{-1}\big(r^{-n}\big)
	\end{equation}
	for all $r>0$, where $C>0$ does not depend on $r$, is satisfied. In the proof of Theorem \ref{AdamsCommFrCharOrl}  we do not use the condition \eqref{adFrCommCharOrl3}.
\end{remark}

If we take $\Phi(t)=t^{p}$ and $\Psi(t)=t^{q}$ in Theorem \ref{AdamsCommFrCharOrl} we get the following corollary.
\begin{corollary}
	Let $1< p < \infty$, $0<\a <n/p$ and $b\in  BMO(\Rn)$. Then $M_{b,\a}$ is bounded from $L^p(\Rn)$ to $L^q(\Rn)$ if and only if $\frac{1}{q}=\frac{1}{p}-\frac{\a}{n}$.
\end{corollary}

The following theorem is valid.
\begin{theorem} \label{CommFrMaxCharOr01}
	Let $0<\a<n$, $b\in  L^1_{\rm loc}(\Rn)$ and $\Phi, \Psi$ be Young functions with $\Phi \in  \Delta_2\cap\nabla_2$ and $\Psi \in  \Delta_2\cap\nabla_2$.
	
	$1.~$ If the condition \eqref{adRieszCharOrl2} holds,
	then the condition $b\in BMO(\Rn)$ is sufficient for the boundedness of $M_{b,\a}$ from $L^{\Phi}(\Rn)$ to $L^{\Psi}(\Rn)$.
	
	$2.~$ If $\Psi^{-1}(t) \lesssim \Phi^{-1}(t)t^{-\a/n}$, then the condition $b\in BMO(\Rn)$ is necessary for the boundedness of $M_{b,\a}$ from $L^{\Phi}(\Rn)$ to $L^{\Psi}(\Rn)$.
	
	$3.~$ If $\Psi^{-1}(t) \thickapprox \Phi^{-1}(t)t^{-\a/n}$ holds,
	then the condition $b\in BMO(\Rn)$ is necessary and sufficient for the boundedness of $M_{b,\a}$ from $L^{\Phi}(\Rn)$ to $L^{\Psi}(\Rn)$.
\end{theorem}
\begin{proof}
	1. The first statement of the theorem follows from the first part of the Theorem \ref{AdamsCommFrCharOrl}.
	
	2. We shall now prove the second part. Suppose that $M_{b,\a}$ is bounded from $L^{\Phi}(\Rn)$ to $L^{\Psi}(\Rn)$. Choose any ball $B=B(x,r)$ in $\Rn$, by \eqref{2.3} and Lemma \ref{charorlc}
	\begin{align*}
		&\frac{1}{|B|} \int_{B}|b(y)-b_{B}|dy  = \frac{1}{|B|} \int_{B} \Big|\frac{1}{|B|} \int_{B} (b(y)-b(z))dz \Big| dy
		\\
		&\le \frac{1}{|B|^2} \int_{B} \int_{B} |b(y)-b(z)|dz dy  = \frac{1}{|B|^{1+\frac{\a}{n}}} \int_{B} \frac{1}{|B|^{1-\frac{\a}{n}}} \int_{B} |b(y)-b(z)| \chi_{_B}(z) dz dy
		\\
		& \le \frac{1}{|B|^{1+\frac{\a}{n}}} \int_{B} M_{b,\a}\big( \chi_{_B}\big)(y) dy \le \frac{2}{|B|^{1+\frac{\a}{n}}} \, \|M_{b,\a}\big( \chi_{_B}\big)\|_{L^{\Psi}}\|\chi_{_B}\|_{L^{\widetilde{\Psi}}}
		\\
		& \le \frac{C}{|B|^{\frac{\a}{n}}} \, \frac{\Psi^{-1}(|B|^{-1})}{\Phi^{-1}(|B|^{-1})} \leq C.
	\end{align*}
	Thus $b\in BMO(\Rn)$.
	
	3. The third statement of the theorem follows from the first and second parts of the theorem.
\end{proof}

\begin{remark}
	Note that, in the case $\Phi(t)=t^{p}$ and $\Psi(t)=t^{q}$ from Theorem \ref{CommFrMaxCharOr01} we get Theorem \ref{rem01} for the operator $M_{b,\a}$.
\end{remark}

The following relations between $[b,M_{\a}]$ and $M_{b,\a}$ are valid (see, for example, \cite{GulDerHasAAMP,ZhWu})\,:

Let $b$ be any non-negative locally integrable function.
$$
|[b,M_{\a}]f(x)| \le M_{b,\a}(f)(x), \qquad x\in\Rn
$$
holds for all $f \in L^1_{\rm loc}(\Rn)$.

If $b$ is any locally integrable function on $\Rn$, then
\begin{align}\label{commaxcom}
	|[b,M_{\a}]f(x)| \leq M_{b,\a}(f)(x)+ 2b^{-}(x)M_{\a}f(x),\qquad x\in\Rn
\end{align}
holds for all $f \in L^1_{\rm loc}(\Rn)$.

By \eqref{commaxcom} and Theorems \ref{AdamsFrMaxCharOrl} and \ref{AdamsCommFrCharOrl} we get the following corollary.
\begin{corollary}\label{OmD01}
	Let $0 < \a<n$, $b\in BMO(\Rn)$, $b^{-}\in L^{\infty}(\Rn)$ and $\Phi,\Psi$ be Young functions with such that $\Phi \in  \Delta_2\cap\nabla_2$ and $\Psi \in  \Delta_2\cap\nabla_2$. If the condition \eqref{adRieszCharOrl2} holds,
	then the operator $[b,M_{\a}]$ is bounded from $L^{\Phi}(\Rn)$ to $L^{\Psi}(\Rn)$.
\end{corollary}

\end{document}